\documentclass[10pt]{article}
\usepackage[margin=1.25in]{geometry}

\usepackage{amsmath, amsthm}
\numberwithin{equation}{section}

\usepackage[english]{babel}
\usepackage[T1]{fontenc}
\usepackage[utf8]{inputenc}
\usepackage{units}

\usepackage{amsmath}
\usepackage{afterpage}

\usepackage{graphicx}
\usepackage{subfigure}
\usepackage{wrapfig}
\usepackage{indentfirst}
\usepackage{latexsym}
\usepackage{sidecap}
\usepackage{booktabs}
\usepackage{amsthm}
\usepackage{amssymb}

\usepackage{setspace}

\usepackage{mathrsfs}
\usepackage{textcomp}
\usepackage{braket}
\usepackage{bbm}
\usepackage{colortbl}
\usepackage{comment}
\usepackage[colorinlistoftodos]{todonotes}

\newcommand{\E}[1]{ \mathbb{E} \left [ #1\right ] }

\newcommand{\RR}{\mathbb{R}}
\newcommand{\NN}{\mathbb{N}}

\newcommand{\Ind}[1]{\mathbbm{1}_{#1}}

\newcommand{\OO}{\mathcal{O}}

\newcommand{\FT}{\left (\mathcal{F}_t\right )_{t \geq 0}-}
\newcommand{\U}{\mathcal{U}}

\usepackage{fancyhdr}
\usepackage{calc} 
 \newtheorem{theorem}{theorem}[section]

 \theoremstyle{definition}
 \newtheorem{definition}[theorem]{Definition}
 \theoremstyle{remark}

 \numberwithin{equation}{section}

\usepackage{comment}
\usepackage[colorinlistoftodos]{todonotes}

\begin{document}

\title{Optimal control for the stochastic FitzHugh-Nagumo model with recovery variable}
\author{Francesco Cordoni$^{a}$ \and Luca Di Persio$^{a}$}
\date{}
\maketitle

\renewcommand{\thefootnote}{\fnsymbol{footnote}}
\footnotetext{{\scriptsize $^{a}$ Department of Computer Science, University of Verona, Strada le Grazie, 15, Verona, 37134, Italy}}
\footnotetext{{\scriptsize E-mail addresses: francescogiuseppe.cordoni@univr.it
(Francesco Cordoni), luca.dipersio@univr.it (Luca Di Persio)}}

\begin{abstract}
In the present paper we derive the existence and uniqueness of a solution for the optimal control problem determined by a stochastic FitzHugh-Nagumo equation with recovery variable. In particular due the cubic non-linearity in the drift coefficients, standard techniques cannot be applied so that the Ekeland's variational principle has to be exploited.
\end{abstract}

\textbf{AMS Classification subjects:} 34K35, 35R60, 49J53, 49K99, 60H15, 65K10, 93E03. \medskip

\textbf{Keywords or phrases: }Stochastic FitzHugh-Nagumo equation, stochastic optimal control, Ekeland's variational principle, stochastic partial differential equations.

\maketitle

\section{Introduction}

The mathematical formulation of the signal propagation in a neural cell has been firstly introduced by A. L. Hodgkin, and A. F. Huxley in \cite{HH}, where  the authors proposed a  mathematical model based on a system of four non-linear, coupled differential equations describing  how action potentials in neurons are initiated and propagated. In particular, latter system describes the evolution in time of four state variables and even if it is possible to state some qualitative properties for it, an analytical solution is missing. Therefore,  alternative approaches have been developed by several authors, as in the case of the celebrated FitzHugh-Nagumo model (FHN), see  \cite{FH,Nag}, where the system is reduced to two equations describing the evolution in time of the (neuronal) voltage variable and of the so called {\it recovery variable}. It is worth to mention that the previous description, as noted by the FitzHugh in his seminal paper, is an example of relaxation oscillator, in fact, FitzHugh referred to his model as the Bonhoeffer–Van der Pol oscillator. 
During recent years, the FHN model has gained a lot of attention, particularly from the point of view of the stochastic analysis in order to consider the influence of random perturbations of  the original, deterministic description, see, e.g. \cite{BMZ,MaDPZ}
In fact, from the experimental point of view, many neuronal activities can be better understood allowing for random components which  affect  the transmission of signals, as well as the inaccuracy of laboratory measures and the lack of a complete knowledge of the particular cerebral activity under study.
Aiming at considering such a generalized, random framework, we will analyse the following stochastic system

{\footnotesize
\begin{equation}\label{EQN:RDFHNIn}
\begin{cases}
\partial_t v(t,\xi) &=\Delta_\xi- I_{ion}(v(t,\xi)) -w(t,\xi) +f(\xi) +\partial_t \beta_1(t)\,,\mbox{ in } [0,T] \times \OO\, ,\\
\partial_t w(t,\xi) &= \gamma v(t,\xi)- \delta w(t,\xi)+ \partial_t \beta_2(t)\,, \mbox{ in } [0,T] \times \OO\, ,\\
\partial_\nu v(t,\xi)&=0\, ,\quad \mbox{ on } [0,T] \times \partial \OO \, ,\\
v(0,\xi) &= v_0(\xi),\quad w(0,\xi)=w_0(\xi)\, ,\mbox{ in } [0,T] \times \OO\, .\\
\end{cases}
\, ,
\end{equation}
}
where, as mentioned above, the variable $v$ represents the voltage quantity, $w$ denotes  the recovery variable, while the other components will be specified in a  while. For the moment, let us note that the function $I_{ion}$ is a polynomial of degree $3$, then standard existence and uniqueness results do not hold for eq. \eqref{EQN:RDFHNIn}, since the non-linear term $I_{ion}$ fails to be Lipschitz continuous. Latter problem is often overcome taking into account some additional regularity properties of the infinitesimal generator,  namely the Laplacian $\Delta$ appearing  in eq. \eqref{EQN:RDFHNIn}, such as the so-called $m-$dissipativity assumption, see, e.g., \cite{Alb3,AlbeverioDPMastrogiacomoPotential,DPZ} and references therein, for details.

We will not concern in the present paper with the existence and uniqueness result, since it is an already established result in literature, but on the existence of an optimal control for the aforementioned equation. In particular in \cite{BCDP}, the existence and uniqueness of an optimal control has been proved for a similar equation, without the recovery variable $w$. To prove the existence of an optimal control in the stochastic case is a rather delicate point and it implies the use  of non trivial results. In particular the main result of the present work, is based, following \cite{BCDP}, on the Ekelands's variational principle.

The present work is so structured, in section \ref{SEC:AS} we introduce the main notation and assumptions used throughout the work, and we state the existence and uniqueness result for the main equation of interest. Then, in section \ref{SEC:OCP}, we derive the main result, namely we prove the existence and uniqueness solution of the optimal control problem associtaed to the FH-N model with recover variable, exploiting the Ekelands's variational principle

\section{The abstract setting}\label{SEC:AS}

Let us consider the following controlled stochastic FitzHugh-Nagumo system of equations
{\footnotesize
\begin{equation}\label{EQN:RDFHN}
\begin{cases}
\partial_t v(t,\xi) &=\Delta_\xi- I_{ion}(v(t,\xi)) -w(t,\xi) +f(\xi) + B_v u(t,\xi) +\partial_t \beta_1(t)\,,\mbox{ in } [0,T] \times \OO\, ,\\
\partial_t w(t,\xi) &= \gamma v(t,\xi)- \delta w(t,\xi)+ \partial_t \beta_2(t)\,, \mbox{ in } [0,T] \times \OO\, ,\\
\partial_\nu v(t,\xi)&=0\, ,\quad \mbox{ on } [0,T] \times \partial \OO \, ,\\
v(0,\xi) &= v_0(\xi),\quad w(0,\xi)=w_0(\xi)\, ,\mbox{ in } [0,T] \times \OO\, .\\
\end{cases}
\, ,
\end{equation}
}
where $v=v(t, \xi)$  represents the transmembrane electrical potential, $w=w(t,\xi)$ is a recovery variable, also known as gating variable and which can be used to describe the potassium conductance, $\OO \subset \RR^d$, $d=2,3,$ is a bounded and open set with smooth boundary $\partial \OO$. Furthermore $\Delta_\xi$ is the Laplacian operator with respect to the spatial variable $\xi$, while $\gamma$ and $\delta$ are positive constants representing phenomenological coefficients, $\nu$ is the outer unit normal direction to the boundary $\partial \OO$ and $\partial_\nu$ denotes the derivative in the direction $\nu$, $f(\xi) \in L^\infty(\OO)$ is a given external forcing term, $I_{ion}$ represents the {\it Ionic current} assumed to be as in the FitzHugh-Nagumo model, namely it is taken as a cubic non-linearity of the following form $I_{ion}(v)= v(v-a)(v-b)$, $v_0$, $w_0 \, \in L^2(\OO)$. and $\beta_1$ and $\beta_2$ two independent $Q_i$-Brownian motions, $i = 1,2$,  $Q_i$ being positive trace class commuting operators. Eventually we assume that the two operators $Q_1$ and $Q_2$ diagonalize on the same basis $\{e_k\}_{k \geq 1}$, namely we assume that there exists a sequence of positive real numbers $\{\lambda_k^i\}_{k \geq 1}$, $i=1,\,2$ such that
\[
Q_i \, e_k = \lambda^i_k \, e_k \, , \quad i=1,\,2\, ,\quad k \geq 1\, ,
\]
moreover we also assume that $Tr Q_i< \infty$, $i = 1, \, 2$. Eventually let $U$ be a Hilbert space equipped with the scalar product $\langle \cdot,\cdot \rangle_U$, we have that $u:[0,T] \to U$ denotes the control and $B_v \in L(U,L^2(\OO))$.

In order to rewrite \eqref{EQN:RDFHN} in a more compact form as an infinite dimensional stochastic evolution equation, let us define the Hilbert space $H := L^2(\OO) \times L^2(\OO)$ endowed with the inner product
\begin{equation}
\left  \langle (v_1,w_1), (v_2,w_2)\right \rangle_H = \gamma  \langle v_1,v_2 \rangle_{2} + \langle w_1,w_2  \rangle_{2} \, ,
\end{equation}
where $\langle \cdot,\cdot \rangle_2$ denotes the usual scalar product in $L^2(\OO)$, and the corresponding norm will be indicated by $|\cdot |_2$. Let us further introduce the space $V:= H^1(\OO) \times L^2(\OO)$ with the norm
\[
|X|_V^2 = \gamma |v|^2_{H^1} + |w|^2_2\, ,\quad X = (v,w) \in H\, .
\]
We then define the operator $A:D(A) \subset H \to H$ as follows
\[
A = 
\begin{pmatrix}
A_0 v  & -w  \\
 \gamma v & -\delta w \\
\end{pmatrix}\, ,\quad 
A_0 =  \Delta_\xi\;,
\]
with domain given by
\begin{equation*}
\begin{split}
&D(A) := D(A_0) \times L^2(\OO)\;, \\ 
&D(A_0)  :=\{u \in H^2(\OO): \partial_\nu u(\xi) =0 \, \mbox{ on }  \partial \OO\},
\end{split}
\end{equation*}

In particular we have that $A$ generates a $C_0-$semigroup satisfying 
\[
\| e^{tA}\| \leq e^{-\omega t} \, ,\quad \omega >0\, ,
\]
see, e.g. \cite{Bon}.
 
We further define the non-linear operator $$F: D(F) := L^6(\OO)\times L^2(\OO) \to H\;,$$ as
\[
F \begin{pmatrix} v \\ w \end{pmatrix}    =
\begin{pmatrix} I_{ion}(v) +f \\ 0 \end{pmatrix}=
\begin{pmatrix} -v(v-a)(v-b) +f \\ 0 \end{pmatrix}.
\]
%We refer to \cite{Dap} for a detailed derivation of the present setting.
%Eventually we make the following assumptions concerning the noise. 
%From the separability of the Hilbert space $H$ we have that it exists an orthonormal basis $\{e_k\}_{k\in \mathbb{N}}$ made of eigenvalues of $A_0$ such that the following bound holds
%\[
%\exists\, M>0, |e_k(\xi)| \leq M,\, \xi \in \OO , \, k \in \mathbb{N}\, .
%\]
In what follows we will assume that it exists a positive constant $\eta$ such that
\[
\left \langle F(x) - F(y) - \eta (x-y) , x-y \right \rangle < 0\, ,\quad x\, , \, y \in H\, ,
\]
and also that it holds $\omega - \eta >0$. This implies that the term $A+F$ is $m-$dissipative in the sense of \cite{Dap}.

Let us thus consider the filtered probability space $\left (\Omega, \mathcal{F}, \mathcal{F}_t, \mathbb{P}\right )$, such that the two independent Wiener processes $\beta_1$ and $\beta_2$ are adapted to the filtration $\mathcal{F}_t$, $\forall\,\, t \geq 0$, and we define $W(t) = (\beta_1(t),\beta_2(t))$ a cylindrical Wiener process on $H$ and by $Q$ the operator
\[
Q = 
\begin{pmatrix}
Q_1  & 0  \\
0 & Q_2 \\
\end{pmatrix}
 \in \mathcal{L}(H;H)\;.
\]
%%Furthermore we have that 
%\[
%\beta_i \in C\left ([0,T]; L^2(\Omega,L^2(\OO))\right ),\quad  i=1,2\, ,
%%\]
%with $\mathcal{L}(\beta_i (t)) \sim \mathcal{N}\left (0, t \sqrt{Q}_i\right )$, $i=1,2$, with $Q_i$ a linear operator on $L^2(\OO)$. We can assume that the operators $Q_i$ are of trace class and that they diagonalize on the same basis $\{e_k\}_{k\in \mathbb{N}}$ of the Hilbert space $H$, namely $Q_i e_k = \lambda_k^i e_k$, $i=1,2$. We also assume that
%$
%\sum_{i=1}^2 \sum_{k=1}^\infty \lambda_k^i < \infty,
%$
%and we denote by $W(t)$  a cylindrical Wiener process on $H$ and 
%we refer to \cite{Dap} for a complete treatment of the topic.
Exploiting  previously introduced notation, eq. \eqref{EQN:RDFHN} can be rewritten as follows
\begin{equation}\label{EQN:FHN}
\begin{cases}
dX(t)=[AX(t)+F(X(t))]dt +  \sqrt{Q}dW(t),\\
X(0) = x_0  \in H\, ,\quad  t \in [0,T]\, ,
\end{cases}   
\, .
\end{equation}

\begin{definition}\label{DEF:Mild}
We say that the function $X \in C_W([0,T];H)$ is called a \textit{mild solution} to \eqref{EQN:FHN} if $X(t): [0,T] \to H$ is continuous $\mathbb{P}-$a.s., $\forall \, t \in [0,T]$ and it satisfies the stochastic integral equation
\[
X(t) = e^{-A t}x + \int_0^t e^{-(t-s)A}\left (-F(s)\right )ds + \int_0^t e^{-(t-s)A}\left (\sqrt{Q}\right ) dW(s), \quad \forall \,\, t \in [0,T]\, .
\]
\end{definition}
%\begin{definition}
%sdfsadfsdfasd
%\end{definition}
The we have the following existence and uniqueness result concerning equation \eqref{EQN:FHN}.
\begin{theorem}\label{THM:E!}
For any $x \in D(F)$, there exists a unique mild solution $X$ to \eqref{EQN:FHN} which satisfies
\[
X \in  L^2_W \left (\Omega;C\left ([0.T];H\right )\right ) \cap  L^2_W \left (\Omega;L^2\left ([0.T];V\right )\right )  \, .
\]
\end{theorem}
\begin{proof}
Under above assumptions the proof follows from \cite[Prop. 3.8]{Alb3} or \cite[theorem 3.1]{Bon}.
\end{proof}

\section{The optimal control problem}\label{SEC:OCP}

Let us now consider a controlled version of equation \eqref{EQN:FHN}. Let then $B \in L\left (U;H\right )$ defined as
\[
B u = \binom{B_v u}{0}\, , \quad B_v \in L(U;L^2(\mathcal{O}))\, .
\]
We shall denote by $\mathcal{U}$ the space of all $\FT$adapted processes $u : [0,T ] \to U$ s.t. $\E{\int_0^T |u(t)|^2_U dt} < \infty$. The space $\mathcal{U}$ is a Hilbert space with the norm $|u|_{\mathcal{U}} = \left (\E{\int_0^T |u(t)|_U^2 dt}\right )^\frac{1}{2}$ and scalar product 
\[
\langle u,v \rangle_{\mathcal{U}} =\left ( \E{\int_0^T \langle u(t),v(t) \rangle_U dt}\right )^\frac{1}{2}, \quad \forall \, u,\, v \in \U\, ,
\]
where $\langle\cdot,\cdot\rangle_U$ is the scalar product of $U$.

Consider the functions $g$, $g_0\, : \, \RR \to \RR$ and $h : U \to \bar \RR := ]-\infty,\infty]$, which satisfy the following conditions
\begin{description}
\item[(i)] $g$, $g_0 \in C^1\left (H\right )$ and $D g$, $D g_0 \in Lip\left (H;H\right )$, where $D$ stands for the Fr\'{e}chet differential\\
\item[(ii)] $h$ is convex, lower-semicontinuous and $\left (\partial h\right )^{-1} \in Lip(U)$ where $\partial h: U \to U$ is the subdifferential of $h$, see, e.g., \cite[p. 82]{Bar4}. Moreover we assume that $\exists$ $\alpha_1>0$ and $\alpha_2 \in \RR$ s.t. $h(u) \geq \alpha_1 |u|^2_U + \alpha_2$, $\forall$ $u \in U$, and we set $L= \|(\partial h)^{-1} \|_{Lip(U)}$.
\end{description}
We consider the following optimal control problem 
\begin{equation}\label{EQN:P}
  \tag{P}
\mbox{Minimize} \,  \E{\int_0^T \left (  g(X(t)) + h(u(t))\right )dt } + \E{g_0(X(T))} \, ,
\end{equation}
subject to $ u \in \mathcal{U} $ and

\begin{equation}\label{EQN:FHNC}
\begin{cases}
dX(t)=[AX(t)+F(X(t))]dt + Bu(t)dt + \sqrt{Q}dW(t)\, ,\\
X(0) = x_0  \in H\, ,\quad  t \in [0,T]\, ,
\end{cases}   
\, .
\end{equation}

%
%In the following we shall assume both \eqref{EQN:2.10A} and $Tr [QA] < \infty$, where $A$ is as above the Laplace operator with domain $H^1_0(\OO) \cap H$.

\begin{theorem}\label{THM:E!4}
Let $x \in D(A)$. Then there exists $C^*>0$ independent of $x$ such that for $LT + \|Dg_0\|_{Lip} < C^*$ there is a unique solution $ \left (u^*,X^*\right ) $ to problem \eqref{EQN:P}.
\end{theorem}
\begin{proof}
Let us consider the function $\Psi: \mathcal{U} \to \bar \RR$ defined by
\[
\Psi(u) = \E{\int_0^T \left (  g(X^u(t)) + h(u(t))\right )dt } + \E{g_0(X^u(T))}\, ,
\]
where $X^u$ is the solution to \eqref{EQN:FHNC}. Recall that $\Psi$ is lower-semicontinuous.

We shall apply Ekeland's variational principle (See, e.g., \cite{Eke} or also \cite{BCDP,Bar3}), that is there is a sequence $ \{u_\epsilon\} \subset \mathcal{U} $ such that
\begin{equation}\label{EQN:2.4}
\begin{split}
\Psi(u_\epsilon) &\leq \inf \{ \Psi(u) \, ; u \in \mathcal{U}\}+ \epsilon\, ,\\
\Psi(u_\epsilon) &\leq \Psi(u) + \sqrt{\epsilon}\left |u_\epsilon - u\right |_{\mathcal{U}} \, ,\quad \forall \, u \in \mathcal{U}\, .
\end{split}
\end{equation}

In other words,
\[
u_\epsilon = \arg \min_{u \in \mathcal{U}} \{  \Psi(u) + \sqrt{\epsilon}\left |u_\epsilon - u\right |_{\mathcal{U}} \}\, .
\]
Hence $ \left (X^{u_\epsilon},u_\epsilon\right ) $ is a solution to the optimal control problem
\begin{equation}\label{EQN:2.5}
\begin{split}
&\min \left \{\E{\int_0^T \left ( g (X^u(t)+ h(u(t))\right ) dt } + \E{g_0\left (X^u(T)\right )}\right .+\\
&\qquad \quad+ \left .\sqrt{\epsilon} \left ( \E{\int_0^T \left |u(t) - u_\epsilon(t) \right |^2_U dt} \right )^{\frac{1}{2}} \, ; u \in \mathcal{U} \right \}\, .
\end{split}
\end{equation}

Equation \eqref{EQN:2.5} means that for all $v \in \mathcal{U}$ and $\lambda >0$ it holds
\[
\begin{split}
& \E{\int_0^T \left ( g(X^{u_\epsilon + \lambda v}(t) + h((u_\epsilon + \lambda v)(t)) \right )dt} + \E{g_0(X^{u_\epsilon + \lambda v}(T))}+ \\
& +\lambda \sqrt{\epsilon}\left ( \E{\int_0^T \left |v(t)\right |^2_{U}dt }\right )^{\frac{1}{2}} \leq\\
&\leq \E{\int_0^T \left (g (X_\epsilon (t) ) + h(u_\epsilon(t)) \right )dt} + \E{g_0(X_\epsilon(T))}\, ,
\end{split}
\]
that is we get
\begin{equation}\label{EQN:21ab}
\begin{split}
&\E{\int_0^T \left \langle Dg(X_\epsilon(t)),Z^v(t) \right \rangle_2 dt} + \E{\int_0^T h'(u_\epsilon(t),v(t))dt} +\\
&+ \E{\left \langle Dg_0(X_\epsilon(T)),Z^v(T) \right \rangle_2} + \sqrt{\epsilon} \left (\E{\int_0^T |v(t)|_U^2 dt}\right )^\frac{1}{2} \leq 0\, ,\quad \forall \, v \in \U\, ,
\end{split}
\end{equation}
where $Z^v$ solves the system in variations associated with \eqref{EQN:FHNC},
\begin{equation}\label{EQN:3.12a}
\begin{cases}
\frac{\partial}{\partial t}Z^v(t)  = A Z^v(t) + DF(X_\epsilon(t))Z^v(t) + Bv(t) \, ,t \in  [0,T] \, ,\\
Z^v(0) = 0 \, ,\\
\end{cases}
\end{equation}
and $h' : U \times U \to \RR$ is the directional derivatives of $h$,  see, e.g., \cite[p.81]{Bar4}, namely 
\[
h'(u_\epsilon,v) = \lim_{\lambda \downarrow 0} \frac{h(u_\epsilon + \lambda v)-h(u_\epsilon)}{\lambda}\, ,\quad \forall \, v \in U\, .
\]

We thus associate with \eqref{EQN:FHNC} the dual stochastic backward equation
\begin{equation}\label{EQN:3.7a}
\begin{cases}
 d p_\epsilon(t) = - \left [A p_\epsilon(t) dt + DF(X_\epsilon)p_\epsilon(t)- D g(X_\epsilon(t)) \right ] dt+ \kappa_\epsilon(t) \sqrt{Q}dW(t) \, , t \in [0,T]\, ,\\
 p_\epsilon(T) =- D g_0(X_\epsilon(T)) \, ,\\
\end{cases}
\, .
\end{equation}
It is well-known that equation \eqref{EQN:3.7a} has a unique solution $(p_\epsilon,\kappa_\epsilon)$ satisfying
\[
\begin{split}
p_\epsilon & \in L^\infty_W\left ([0,T];H \right ) \cap L^2_W \left ([0,T];V\right )\, ,\\
k_\epsilon & \in L^2_W\left ([0,T]; H\right ) \, ,
\end{split}
\]
(See, e.g., \cite[Prop. 4.2]{FT} or \cite{Tes}). By It\^{o}'s formula we have from \eqref{EQN:3.12a} and \eqref{EQN:3.7a} that
\[
d \left \langle p_\epsilon, Z^v \right \rangle_H = \left \langle d p_\epsilon, Z^v \right \rangle_H + \left \langle p_\epsilon, d Z^v \right \rangle_H\, ,
\]
and this immediately implies
\[
\E{\int_0^T \left \langle Dg(X_\epsilon(t)),Z^v(t) \right \rangle_H dt} + \E{\left \langle Dg_0(X_\epsilon(T)),Z^v(T) \right \rangle_H} = \E{\int_0^T \left \langle Bv(t),p_\epsilon(t) \right \rangle_H dt}\, ,
\]
which substituted in \eqref{EQN:21ab} yields that $\forall$ $v \in \U$, the following inequality holds
\[
\begin{split}
& \E{\int_0^T h'(u_\epsilon(t),v(t))dt}+\sqrt{\epsilon} \left (\E{\int_0^T |v(t)|_U^2 dt}\right )^\frac{1}{2} \leq \\
& \leq \E{\int_0^T \left \langle B^* p_\epsilon (t),v(t) \right \rangle_U dt} \, .
\end{split}
\]
Let $G(u):= \E{\int_0^T h(u(t))dt}$, then its sub-differential $\partial G: \mathcal{U} \to \U$, evaluated in $u_\epsilon$  is given by 
\[
\partial G(u_\epsilon)= \left \{v^* \in \U \, : \, \langle v,v^* \rangle_{\U} \leq \E{\int_0^T h'(u_\epsilon(t),v(t))dt}\, , \, \forall \, v \in \U \right \}\, .
\]
(See, e.g., \cite[p.81]{Bar4}). Then we infer that 
\[
u_\epsilon(t) = (\partial h)^{-1}\left (B^* p_\epsilon (t) + \sqrt{\epsilon} \tilde{\theta}_\epsilon\right )\, , \, \, t \in [0,T]\, , \quad \mathbb{P}-a.s. \, ,
\]
where $\tilde{\theta}_\epsilon \in \U$ and $|\tilde{\theta}_\epsilon|_{\U} \leq 1$, $\forall \, \epsilon > 0$.

Therefore, we have shown that
\begin{equation}\label{EQN:2.6c}
\begin{split}
& u_\epsilon = (\partial h)^{-1} \left (B^* p_\epsilon + \theta_\epsilon\right )\, , \|\theta_\epsilon \|_{L^2\left ([0,T] \times \Omega ;U\right )} \leq \sqrt{\epsilon}\, ,\\
& d p_\epsilon(t) = - \left [A p_\epsilon(t) dt + DF(X_\epsilon)p_\epsilon(t)- D g(X_\epsilon(t)) \right ] dt+ \kappa_\epsilon(t) \sqrt{Q}dW(t) \, , t \in [0,T]\, ,\\
& p_\epsilon(T) = -D g_0(X_\epsilon(T)) \, ,\\
\end{split}
\, .
\end{equation}

Using the It\^{o} formula applied to $|X|^2_2$, we have that $\forall$ $\epsilon >0$ it holds
\begin{equation}\label{EQN:ItoFHN}
\begin{split}
\left | X_\epsilon(t) \right |^2_H &= |x|^2_H + 2 \int_0^t \left \langle AX_\epsilon(s) + F(X_\epsilon(s)) + B u_\epsilon(s), X_\epsilon(s) \right \rangle_H ds+ \\
&+ Tr Q t + 2 \int_0^t \left \langle X_\epsilon(s), \sqrt{Q} dW(s)\right \rangle_H\, .
\end{split}
\end{equation}

(Here and everywhere in the following we shall denote by $C$ several positive constants independent of $\epsilon$.)

From the fact that $\left \langle X_\epsilon(s), \sqrt{Q} dW(s)\right \rangle_H$ is a square integrable martingale, \cite[Th. 3.14, Th. 4.12]{Dap} and recalling the assumption $Tr AQ < \infty$ we have that
\[
\E{\sup_{t \in [0,T]} \left |\int_0^t \left \langle X_\epsilon(s), \sqrt{Q} dW(s)\right \rangle_H \right |} \leq C \E{\int_0^T |X_\epsilon (t) |^2_H dt} \, ,
\]
and from the fact that $A$ generates a strongly continuous semigroup, see, e.g. \cite{Bon}, we have that
\[
\begin{split}
&\int_0^t \left \langle AX_\epsilon(s),X_\epsilon(s) \right \rangle_H ds \leq C_1 \int_0^t |X_\epsilon (s)|^2_V ds\, .\\
\end{split}
\]
We also have that it holds,
\[
\int_0^t \left \langle F(X_\epsilon(s)) ,X_\epsilon(s) \right \rangle_H ds \leq C |X_\epsilon(t)|^2_H\, ,\\
\]
see, e.g. \cite{Alb3,Bon} for details. Eventually from assumption (ii) we have
\[
\int_0^t \left \langle B u(s),X_\epsilon(s) \right \rangle_H ds \leq L^{-1} \int_0^T |u_\epsilon(t)|^2_U dt\, .
\]
Taking then the expectation on both side of \eqref{EQN:ItoFHN} yields
\[
\E{\sup_{t \in [0,T]}\left | X_\epsilon(t) \right |^2_H} + \E{\int_0^T |X_\epsilon (t)|^2_V dt} \leq C_1+ C_2 \int_0^T \E{\sup_{s \in [0,t]} \left |X_\epsilon (s)\right |^2_H dt}
\]
and applying Gronwall's lemma it follows eventually that
\begin{equation}\label{EQN:2.10}
\E{\sup_{t \in [0,T]} \left |X_\epsilon (t)\right |^2_H } + \E{\int_0^T \left |X_\epsilon (t)\right |^2_{V} dt} \leq C (1+ |x|^2_H) \, .
\end{equation}

%If we now apply the It\^{o} formula to the function $X \to \frac{1}{2}|X|^2_{V}$, taking into account that $Tr[QA]<\infty$ and proceeding as above we have that 
%\begin{equation}\label{EQN:2.10C}
%\E{\sup_{t \in [0,T]} |X_\epsilon (t)|^2_{V}} + \E{\int_0^T \left |X_\epsilon (t)\right |^2_{V} dt} \leq C(1+ |x|^2_{V})\, .
%\end{equation}

In an analogous manner, applying It\^{o} formula to $|p_\epsilon|^2_H$ by \eqref{EQN:2.6c} we obtain that

\[
\begin{split}
&\frac{1}{2}d |p_\epsilon (t)|_H^2 =-\left \langle A p_\epsilon (t)+ DF(X_\epsilon(t))p_\epsilon(t)- Dg(X_\epsilon(t)), p_\epsilon(t) \right \rangle_H + \\
&= \frac{1}{2} \left \langle \kappa_\epsilon(t),\kappa_\epsilon(t) \right \rangle_H dt + \left \langle p_\epsilon(t),\kappa_\epsilon(t) \sqrt{Q}d W(t) \right \rangle_H \, .
\end{split}
\]
which yields after applying arguments similar to the ones above 
\begin{equation}\label{EQN:2.11}
\begin{split}
&\E{\sup_{t \in [0,T]} |p_\epsilon (t)|^2_H} + \E{\int_0^T |p_\epsilon (t)|^2_{V}\, dt} + \E{\int_0^T |\kappa_\epsilon(t)|_H^2 dt}  \leq \\
&\leq C + \E{\left |X_\epsilon(T)\right |^2_H} \leq C \, ,\quad \forall \, \epsilon > 0\, . 
\end{split}
\end{equation}

We have that
\begin{equation}
\begin{split}
&\frac{\partial}{\partial t} \left (X_\epsilon(t) -X_\lambda(t)\right ) =A \left (X_\epsilon(t) -X_\lambda(t)\right ) +\left (F\left (X_\epsilon(t) \right ) -F \left (X_\lambda (t)\right )\right )  +\\
&+BB^*( p_\epsilon(t)-p_\lambda(t)) + B(\theta_\epsilon(t) - \theta_\lambda(t ))\, .
\end{split}
\end{equation}
In virtue of \eqref{EQN:2.11} this yields
\[
\begin{split}
&\frac{1}{2} \left |X_\epsilon(t) -X_\lambda(t)\right |^2_H + \int_0^t\left |X_\epsilon(s) -X_\lambda(s)\right |_V^2 ds \leq\\
&\leq  \int_0^t \left \langle F\left (X_\epsilon(s)\right )  -F \left (X_\lambda (s)\right ),X_\epsilon(s)-X_\lambda (s)\right \rangle_H \, d s \\
&\quad+ L\int_0^t |p_\epsilon(s) -p_\lambda(s)|_H|X_\epsilon (s)-X_\lambda (s)|_H ds  \\
&\quad+C  \int_0^t |\theta_\epsilon(s) -\theta_\lambda(s)|_U|X_\epsilon (s)-X_\lambda (s)|_H ds \, , \quad \forall \, t \in [0,T]\, ,
\end{split}
\]
where $L = \|(\partial h)^{-1}\|_{Lip}$.

We further have that, see, e.g. \cite{Alb3,Bon}
\[
\left \langle F(X_\epsilon)-F(X_\lambda),X_\epsilon - X_\lambda\right \rangle_H \leq C \left |X_\epsilon - X_\lambda \right |^2_H\, ,
\]
%which yields
%then 
%{\small
%\[
%\begin{split}
%- \int_0^t \int_{\OO} \left (f(X_\epsilon) -f(X_\lambda)\right )(X_\epsilon -X_\lambda) \, d\xi \, ds \leq
% C \int_0^t |X_\epsilon(s) - X_\lambda(s)|^2_2 ds\, , \quad \forall \, \epsilon , \, \lambda \, >0 \, ,
%\end{split}
%\]
%}
which yields, for $t \in [0,T]$, applying Young inequality,
\begin{equation}\label{EQN:2.13}
\begin{split}
&\left |X_\epsilon(t) -X_\lambda(t)\right |^2_2 + \int_0^t\left |X_\epsilon(s) -X_\lambda(s)\right |_{V}^2 ds \leq \\
&\leq  C\left ( L\int_0^t |p_\epsilon (s) -p_\lambda(s) |^2_2 ds+ \int_0^t |X_\epsilon(s) - X_\lambda(s) |_H^2 ds +\epsilon + \lambda \right )\, .
\end{split}
\end{equation}

Applying Gronwall's lemma in \eqref{EQN:2.13}, we have
\begin{equation}\label{EQN:2.16}
\begin{split}
&\left |X_\epsilon(t) -X_\lambda(t)\right |^2_2 + \int_0^t\left |X_\epsilon(s) -X_\lambda(s)\right |_{V}^2 ds \leq \\
&\leq  C \left (L\int_0^T |p_\epsilon (s) -p_\lambda(s) |^2_2 ds+ \epsilon + \lambda \right )  \, , \quad \forall \, \epsilon\, , \, \lambda \, > 0\, , \, t \in [0,T]\, .
\end{split}
\end{equation}

Similarly we get by the It\^{o} formula 
\begin{equation}\label{EQN:2.17}
\begin{split}
&\left |p_\epsilon(t) -p_\lambda(t)\right |^2_H + \int_t^T\left |p_\epsilon(s) -p_\lambda(s)\right |_{V}^2 ds + \frac{1}{2} \int_t^T |\kappa_\epsilon(s) - \kappa_\lambda(s) |^2_H ds =  \\
&= |D g_0(X_\epsilon(T)) - D g_0 (X_\lambda(T))|^2_H+\\
&\quad+\int_t^T  \left \langle DF(X_\epsilon(s))p_\epsilon(s) - DF(X_\lambda(s))p_\lambda(s), p_\epsilon(s)-p_\lambda(s)\right \rangle_H ds +\\
&\quad- \int_t^T \left  \langle \kappa_\epsilon(s) - \kappa_\lambda(s))\sqrt{Q} dW(s) , X_\epsilon(s) - X_\lambda(s) \right \rangle_H  \leq \\
&=\int_t^T \left \langle DF(X_\epsilon(s))(p_\epsilon(s)-p_\lambda(s)),p_\epsilon(s)-p_\lambda(s) \right \rangle \, ds +\\
&\quad+ \int_t^T  \left \langle  p_\lambda (s) (DF(X_\epsilon(s))- DF(X_\lambda(s))),p_\epsilon(s)-p_\lambda(s)\right \rangle_H  \, ds +\\
& \quad + \int_t^T \left  \langle \kappa_\epsilon(s)- \kappa_\lambda(s))\sqrt{Q} dW(s) , X_\epsilon(s) - X_\lambda(s) \right \rangle_H +\\
&\quad+ |D g_0\left (X_\epsilon(T)\right )- D g_0\left (X_\lambda(T)\right )|_H^2\leq \\
&\leq C \left ( \int_t^T (|X_\epsilon (s)|^2_H +1) |p_\epsilon(s)-p_\lambda(s)|_H^2 \, ds\right )+\\
& \quad+C \left (\int_t^T \left (1+|X_\epsilon (s)|^2 + |X_\lambda (s)|^2\right ) | X_\epsilon(s)- X_\lambda(s)|_H |p_\epsilon(s)-p_\lambda(s)|_H |p_\epsilon (s) |_H \, ds \right )+\\
& \quad + \int_t^T \left  \langle \kappa_\epsilon(s) - \kappa_\lambda(s)) \sqrt{Q} dW(s) , X_\epsilon(s) - X_\lambda(s) \right \rangle_H+ \\
& \quad+ \|Dg_0\|_{Lip} |X_\epsilon(T) - X_\lambda(T)|_H^2  \, , \quad t \in [0,T]\, , \mathbb{P}-a.s. \, .
\end{split}
\end{equation}

%Proceeding as above, we also have
%\begin{equation}\label{EQN:2.18}
%\begin{split}
%&\int_{\OO} |X_\epsilon (s)||p_\epsilon (s) -p_\lambda (s)|^2 d \xi \leq |p_\epsilon (s) - p_\lambda (s)|_4 |p_\epsilon (s) - p_\lambda (s)|_2 |X_\epsilon(s)|_4 \leq \\
%&\leq \frac{1}{2} |p_\epsilon(s) -p_\lambda(s)|^2_{H^1_0(\OO) } + \frac{1}{2} |p_\epsilon(s) -p_\lambda(s)|^2_{2}|X_\epsilon(s)|_4^2 \, .
%\end{split}
%\end{equation}

Exploiting again Young's inequality, and denoting for short 
\[
T_{\epsilon,\lambda} := (1+|X_\epsilon|^2_H + |X_\lambda|^2_H) |p_\epsilon|_H  \, ,
\]
we get,
\begin{equation}\label{EQN:2.19}
\begin{split}
& \left (| X_\epsilon(s)- X_\lambda(s)|_H |p_\epsilon(s)-p_\lambda(s)|_H \right )T_{\epsilon,\lambda} \leq \\
&\leq C \left ( |X_\epsilon - X_\lambda|_H^{2} + \left | p_\epsilon -p_\lambda \right |_H^{2} \right ) T_{\epsilon,\lambda}\, .
\end{split}
\end{equation}

Substituting now \eqref{EQN:2.19} into \eqref{EQN:2.13}, \eqref{EQN:2.17}, we obtain $\mathbb{P}-$a.s.
\begin{equation}\label{EQN:2.18b}
\begin{split}
&  |X_\epsilon(t) - X_\lambda(t)|^2_H + \left |p_\epsilon(t) -p_\lambda(t)\right |_H^2 + \int_0^t |X_\epsilon(s) - X_\lambda(s)|^2_{V} ds +\\
&\quad+ \int_t^T \left |p_\epsilon(s) -p_\lambda(s)\right |_{V}^2 ds + \int_t^T |\kappa_\epsilon(s) -\kappa_\lambda(s)|_H^2 ds \leq \\
& \leq C \left ( L\int_0^t \left |p_\epsilon(s) -p_\lambda(s)\right |_{H}^2 ds + \epsilon + \lambda  \right ) + C \int_t^T \left |p_\epsilon(s) -p_\lambda(s)\right |_{2}^2 |X_\epsilon(s)|^2_H  ds+ \\
&\quad+ \|Dg_0\|_{Lip} |X_\epsilon(T)-X_\lambda(T)|^2_2+ \\
&\quad+C\int_t^T \left ( |X_\epsilon (s) - X_\lambda (s)|_H^{2} + \left |p_\epsilon (s) -p_\lambda (s) \right |_H^{2} \right ) T_{\epsilon,\lambda} (s) ds+\\
& \quad- \int_t^T \left  \langle \kappa_\epsilon(s) - \kappa_\lambda(s))\sqrt{Q} dW(s) , X_\epsilon(s) - X_\lambda(s) \right \rangle_H \, , \quad \forall \, t \in [0,T]\, .
\end{split}
\end{equation}

Exploiting thus the fact that the process $r \mapsto \int_t^r \left \langle (\kappa_\epsilon -\kappa_\lambda)\sqrt{Q} dW(s),X_\epsilon(s)-X_\lambda(s)\right \rangle_2$ is a local martingale on $[t,T]$, hence by the Burkholder-Davis-Gundy inequality, see, e.g., \cite[p.58]{DapE}, we have for all $r \in [t,T]$
\begin{equation}\label{EQN:BDG}
\begin{split}
& \E{\sup_{r \in [t,T]} \left |\int_t^r \left \langle (\kappa_\epsilon(s)-\kappa_\lambda(s)) \sqrt{Q} dW(s),X_\epsilon(s) -X_\lambda(s) \right \rangle_H \right  |}\leq \\
&\leq C\left (\E{\int_0^r |\kappa_\epsilon(s)-\kappa_\lambda(s)|^2_H |X_\epsilon(s) - X_\lambda(s)|_H^2 ds}\right )^\frac{1}{2} \leq \\
&\leq C \E{\sup_{s \in [t,r]}|X_\epsilon(s) -X_\lambda(s)|^2_H} + C \E{\int_t^r |\kappa_\epsilon (s) -\kappa_\lambda(s)|^2_H ds }\, .
\end{split}
\end{equation}

Taking then the expectation in and by \eqref{EQN:2.18b}, and using \eqref{EQN:BDG} we get
\begin{equation}\label{EQN:2.19b}
\begin{split}
&\E{\sup_{s \in [t,T]} \left (|X_\epsilon (s) -X_\lambda (s)|^2_H +  |p_\epsilon (s) -p_\lambda (s)|^2_H \right )} \\
&\quad+\E{ \int_0^T |X_\epsilon (s) -X_\lambda (s)|^2_{V} ds+ \int_t^T |p_\epsilon (s) -p_\lambda (s)|^2_{H} ds } \\
&\quad+ \E{\int_t^T |\kappa_\epsilon (s) - \kappa_\lambda (s)|_H^2 ds} \leq \\
&\leq \|Dg_0\|\E{|X_\epsilon(T)-X_\lambda(T)|^2_H}+ C \left ( L\E{\int_0^T |p_\epsilon(s) -p_\lambda(s)|^2_H\,ds}+\epsilon + \lambda\right )\\
&\quad+ C \E{\sup_{s \in [t,T]}|X_\epsilon(s) -X_\lambda(s)|_H^2} \\
&\quad+C  \E{ \int_t^T \left (|p_\epsilon (s) -p_\lambda (s)|^2_{H} + |X_\epsilon(s)-X_\lambda(s)|_H^2\right )\left (|X_\epsilon(s)|^2_H + T_{\epsilon,\lambda}(s) \right ) ds } \, . \\
%&\leq C\left (\epsilon + \lambda + \E{\int_0^t |X_\epsilon(s) - X_\lambda(s)|_H^2 ds} + \E{ \int_t^T |p_\epsilon (s) -p_\lambda (s)|^2_{H}\left (|X_\epsilon(s)|^2_H   + T_{\epsilon,\lambda} \right ) ds }\right ) \, .
\end{split}
\end{equation}
Taking into account estimates \eqref{EQN:2.16} and \eqref{EQN:2.17}, from \eqref{EQN:2.19b} we have 
\begin{equation}\label{EQN:2.19g}
\begin{split}
&\E{\sup_{s \in [t,T]} \left (|X_\epsilon (s) -X_\lambda (s)|^2_H +  |p_\epsilon (s) -p_\lambda (s)|^2_H \right )} \\
&\quad+\E{ \int_0^T |X_\epsilon (s) -X_\lambda (s)|^2_{V} ds+ \int_t^T |p_\epsilon (s) -p_\lambda (s)|^2_{H} ds } \\
&\quad+ \E{\int_t^T |\kappa_\epsilon (s) - \kappa_\lambda (s)|_H^2 ds} \leq \\
&\leq \tilde{C} \left ( L\E{\int_0^T |p_\epsilon(s) -p_\lambda(s)|^2_H\,ds}\right )  \\
&\quad+ \tilde{C} \left (\E{ \int_t^T |p_\epsilon (s) -p_\lambda (s)|^2_{H} \left (|X_\epsilon(s)|^3_H   + T_{\epsilon,\lambda}(s) \right ) ds }\right )\\
&\quad+ \tilde{C} \|D g_0\|_{Lip} \E{|X_\epsilon(T)-X_\lambda(T)|^2_H} + \tilde{C}(\epsilon + \lambda) \, . \\
%&\leq C\left (\epsilon + \lambda + \E{\int_0^t |X_\epsilon(s) - X_\lambda(s)|_H^2ds} + \E{ \int_t^T |p_\epsilon (s) -p_\lambda (s)|^2_{H}\left (|X_\epsilon(s)|^2_H   + T_{\epsilon,\lambda} \right ) ds }\right ) \, ,
\end{split}
\end{equation}
where $\tilde{C}$ is a positive constant independent of $\epsilon$ and $\lambda$. It follows that if $\tilde{C}(LT + \| D g_0 \|_{Lip})<1$, then, for any $t \in [0,T]$,
\begin{equation}\label{EQN:2.19h}
\begin{split}
&\E{\sup_{s \in [t,T]} \left (|X_\epsilon (s) -X_\lambda (s)|^2_H +  |p_\epsilon (s) -p_\lambda (s)|^2_H \right )} \\
&\quad+\E{ \int_0^T |X_\epsilon (s) -X_\lambda (s)|^2_{V } ds+ \int_t^T |p_\epsilon (s) -p_\lambda (s)|^2_{H} ds } \\
&\quad+ \E{\int_t^T |\kappa_\epsilon (s) - \kappa_\lambda (s)|_H^2 ds} \leq \\
&\leq C  \E{ \int_t^T |p_\epsilon (s) -p_\lambda (s)|^2_{H} \left (|X_\epsilon(s)|^2_H + T_{\epsilon,\lambda}(s) \right ) ds }+ C(\epsilon + \lambda) \, . 
\end{split}
\end{equation}
Let us define for $j \in \NN$ 
\[
\Omega_j := \left \{ \omega \in \Omega \, : \sup_{\epsilon}\sup_{t \in [0,T]} \left (|X_\epsilon(t)|_H^2 +|X_\epsilon(t)|_{V}^2+ |p_\epsilon(t)|^2_H \right ) dt \leq j \right \} \, ,
\]
then  estimates \eqref{EQN:2.10} implies that
\[
\mathbb{P} \left (\Omega_j\right ) \geq 1- \frac{C}{j}\, ,\quad  \forall \, j \in \mathbb{N}\,,
\]
for some constant $C$ independent of $\epsilon$.

If we set $ X^j_\epsilon := \Ind{\Omega_j} X_\epsilon $, $p^j_\epsilon := \Ind{\Omega_j} p_\epsilon$ and $\kappa^j_\epsilon := \Ind{\Omega_j} \kappa_\epsilon$, then %Clearly $X^j_\epsilon$, $p_\epsilon^j$ and $\kappa^j_\epsilon$ 
such quantities satisfy the system \eqref{EQN:2.6c}, with $\Ind{\Omega_j}\sqrt{Q} d W$. The latter means that estimate \eqref{EQN:2.19h} still holds in this context, so that we have

\begin{equation}\label{EQN:2.22}
\begin{split}
&\E{\sup_{s \in [t,T]} |X_\epsilon^j (s) -X_\lambda^j (s)|^2_H + \sup_{s \in [t,T]} |p_\epsilon^j (t) -p_\lambda^j (t)|^2_H } \\
&\quad+\E{ \int_t^T |p_\epsilon^j (s) -p_\lambda^j (s)|^2_{V} ds } + \E{\int_t^T |(\kappa_\epsilon (s) - \kappa_\lambda (s))\chi_j |_H^2 ds} \leq \\
&\leq C_j \int_t^T \E{ |p_\epsilon^j(s) - p_\lambda^j(s)|^2_H} ds+ C\left (\epsilon + \lambda\right )\, , \quad j \in \mathbb{N}\, .
\end{split}
\end{equation}
By Gronwall's lemma we get, for any $t \in [0,T]$
\begin{equation}\label{EQN:2.23}
\E{\sup_{s \in [t,T]} |X_\epsilon^j (s) -X_\lambda^j (s)|^2_H + \sup_{s \in [t,T]} |p_\epsilon^j (s) -p_\lambda^j (s)|^2_H } \leq C(\epsilon+\lambda)e^{C_j T}\,  ,
\end{equation}
hence, for $ \epsilon \to 0 $ and all $ j \in \mathbb{N} $ and all $t \in [0,T]$, we obtain 
\begin{equation}\label{EQN:2.24}
\begin{split}
&X^j_\epsilon \to X^j \quad \mbox{ in } \quad L^2\left ( \Omega_j ; L^2([0,T] \times \OO)\times L^2([0,T] \times \OO)\right )  \, ,\\
&p^j_\epsilon \to p^j \quad \mbox{ in } \quad L^2\left ( \Omega_j ; L^2([0,T] \times \OO)\times L^2([0,T] \times \OO)\right )\, .
\end{split}
\end{equation}

Therefore for each $ \omega \in \Omega $, we have that $ \left \{ X_\epsilon(t,\omega), p_\epsilon(t,\omega) \right \} $ are Cauchy sequences in $  L^2\left ([0,T] \times \OO \right ) $, with respect to $\epsilon$ and by estimates \eqref{EQN:2.10} and \eqref{EQN:2.11} it follows that taking related  subsequences, still denoted  by $\epsilon $, we have 
\begin{equation}\label{EQN:2.25}
\begin{aligned}
 &X_\epsilon \rightharpoonup X^* && \mbox{ in } L^2\left ([0,T] \times \Omega ; V\right ) \, ,\\
% &p_\epsilon \rightharpoonup p^* && \mbox{ in }  L^\infty \left ([0,T] ; L^2\left (\Omega \times \OO\right )\right ) \, ,\\
 &p_\epsilon \rightharpoonup p^* && \mbox{ in }  L^2\left ([0,T] \times \Omega \times \OO \times \OO \right ) \, ,\\
 &p_\epsilon \rightharpoonup p^* && \mbox{ in } L^2\left ([0,T] \times \Omega ; V\right ) \, ,\\
 &u_\epsilon \rightharpoonup u^* && \mbox{ in }  L^\infty \left ([0,T] ; L^2\left (\Omega \times U\right )\right ) \, ,\\
\end{aligned}
\end{equation}
where $\rightharpoonup$ means weak (respectively, weak-star) convergence, so we have for $ \epsilon \to 0 $
\begin{equation}\label{EQN:2.26}
X_\epsilon \to X^* \, ,\quad p_\epsilon \to p^* \, , a.e. \, \mbox{ in } [0,T] \times \Omega \times \OO \times \OO \, .
\end{equation}

We also have, since $ \left \{ I_{ion}\left (v_\epsilon\right )\right \} $ is bounded in $ L^{\frac{4}{3}} \left ([0,T]\times \Omega \times \OO \right )$, then it is weakly compact in $ L^1\left ([0,T]\times \Omega \times \OO \right ) $ and by \eqref{EQN:2.26} we have that for a subsequence $\{\epsilon\}\to 0$,
\[
I_{ion}\left (v_\epsilon\right ) \to I_{ion}(v^*) \,, \quad a.e. \, \mbox{ in } \, [0,T] \times \Omega \times \OO\, ,
\]
which implies that 
\begin{equation}\label{EQN:2.27}
I_{ion}\left (v_\epsilon\right ) \to I_{ion}(v^*) \, \quad \, \mbox{ in } \, L^1\left ([0,T] \times \Omega \times \OO\right ) \, .
\end{equation}
Then, letting $ \epsilon \to 0 $ we obtain 
\[
\begin{cases}
dX^*(t) =A X^*(t) dt+F(X^*(t)) dt +\sqrt{Q} dW(t)+ Bu^* (t)dt\,, t \in [0,T]\, ,\\
X^*(0) = x \, ,
\end{cases}
\, .
\]
Taking into account that $ \Psi $ is weakly lower semicontinuous in $ \U $ we infer by \eqref{EQN:2.4} that 
\[
\Psi(u^*) = \inf \left \{ \Psi(u); u \in \U \right \}\, ,
\]
therefore $ \left (X^*,u^*\right ) $ is optimal for the problem \eqref{EQN:P} and the proof of existence is therefore complete.
\end{proof}
Concerning the uniqueness for the optimal pair $(X^*,u^*)$ given by Th. \ref{THM:E!4}, we have that it follows by the same argument via the maximum principle result for problem \eqref{EQN:P}, namely one has the following result.

\begin{theorem}\label{THM:4.2}
Let $\left (X^*,u^*\right )$ be optimal in problem \eqref{EQN:P}, then
\begin{equation}\label{EQN:3.32a}
u^* =(\partial h)^{-1} (B^* p)\, , a.e. \: t \in [0,T]\;,
\end{equation}
where $p$ is the solution to the backward stochastic equation \eqref{EQN:3.7a}.
\end{theorem}
\begin{proof}
%{\bf to check\\}
If $(X^*,u^*)$ is optimal for the problem \eqref{EQN:P}, then  by the same argument used to prove Th.  \ref{THM:E!4}, see \eqref{EQN:21ab}, we have
\begin{equation}\label{EQN:21abc}
\begin{split}
&\E{\int_0^T \left \langle Dg(X^*(t)),Z^v(t) \right \rangle_2 dt} + \E{\int_0^T h'(u^*(t),v(t))dt} \\
&+ \E{\left \langle Dg_0(X^*(T)),Z^v(T) \right \rangle_2}  \leq 0\, ,\quad \forall \, v \in \U\, ,
\end{split}
\end{equation}
where $Z^v$ is solution to equation \eqref{EQN:3.12a} with $X_\epsilon$  replaced by $X^*$.
This implies as above that \eqref{EQN:3.32a} holds.
\end{proof}

\begin{proof}[The uniqueness in \eqref{EQN:P}]
If $(X^*,u^*)$ is optimal in \eqref{EQN:P} then it satisfies systems \eqref{EQN:FHN}, \eqref{EQN:3.32a} and \eqref{EQN:21abc}, so that arguing as in the proof of Th. \ref{THM:E!4}, the same set of estimates implies that the previous  system has at most one solution if $LT+ \|D g_0 \|_{Lip}< C^*$, where $C^*$ is sufficiently small.
\end{proof}

\section{Conclusions}
In the present work we have derived the existence and uniqueness of the solution to the control problem associated to a FH-N system of equations perturbed by a Gaussian noise and with respect to a recovery variable. We would like to underline that the presented result has potential applications in medicine, particularly from the point of view of neuronal diseases care. Indeed, the scheme of equations we have studied is linked to the Bonhoeffer–van der Pol oscillator, namely a  nonlinear damping governed by a second-order differential equation that we are able to treat in presence of random (Gaussian) noise. The latter aspect is of great relevance in desincronize abnormal electrical activities that happen under the influence of pathologies as the Parkinson's one, or during epileptic attacks.
Possible generalizations of the proposed analysis will concern the study of the full Hodgkin-Huxley model, when a random source of noise has to be taken into consideration, as well as the study of the aforementioned models over networks of interconnected neurons, mainly following the approach derived in \cite{CordoniDPJMAA2017,CDP2017}. The latter are the subjects of our ongoing research.

\section*{Acknowledgement}
The authors wish to thank Prof. Viorel Barbu for his stimulating comments and enlightening suggestions. The authors would like also to thank the group \textit{Gruppo Nazionale per l'Analisi Matematica, la Probabilità e le loro Applicazioni} (GNAMPA) for the financial support that has founded the present research within the project \textit{Stochastic Partial Differential Equations and Stochastic Optimal Transport with Applications to Mathematical Finance}.

\end{document}